%% file: CSF-1-WL_v3.tex
\documentclass[11pt,reqno]{amsart}
\usepackage{amsmath,amssymb,amsthm,mathtools}
\usepackage{booktabs,longtable}
\usepackage{graphicx}
\newtheorem{theorem}{Theorem}[section]
\newtheorem{lemma}[theorem]{Lemma}
\newtheorem{proposition}[theorem]{Proposition}
\newtheorem{corollary}[theorem]{Corollary}
\theoremstyle{definition}

\theoremstyle{remark}

\newcommand{\wl}{\mathrm{1WL}}
\newcommand{\de}{d_E^{\cong}}
\newcommand{\Dcr}{D_{\mathrm{CR}}^{\mathrm{conn}}}
\newcommand{\core}{\operatorname{core}}

\title{Sharp Bounds for Common Spanning Forests under Color Refinement}
\author{Jianshen Zhu}
\address{Department of Information Sciences and Technology, Tokyo University of Science, Noda, Chiba 278-8510, Japan}
\email{zhu\_js@rs.tus.ac.jp}
\date{\today}

\begin{document}
\begin{abstract}
We study the common spanning structure of connected finite
simple graphs that are indistinguishable by uniformly
initialized color refinement. We prove that any two such
graphs with common cyclomatic number $r\ge2$ contain
isomorphic spanning forests with at most $r-1$ components.
This bound depends only on $r$ and is sharp for every
$r\ge2$, even when both graphs have maximum degree
at most three. We also determine the maximum, over all such graph
pairs without restricting their order, of the minimum edge
symmetric difference under vertex bijections: it is exactly
$4(r-1)$.
\end{abstract}
\maketitle
\pagestyle{plain}
\raggedbottom

\input{sections/01_introduction}

\input{sections/02_preliminaries}
\input{sections/03_common_forests}

\input{sections/04_sharpness}

\bibliographystyle{plain}
\bibliography{CSF-1-WL}

\appendix
\input{appendices/A_subcubic_tree_induction}
\end{document}

%% file: sections/01_introduction.tex
\section{Introduction}\label{sec:introduction}

Color refinement is a standard tool in graph isomorphism testing:
differences in color multiplicities certify that two graphs are
nonisomorphic. However, color refinement does not distinguish every
pair of nonisomorphic graphs. This limitation also applies to
standard message-passing graph neural networks with identical
initial node features, whose distinguishing power is bounded by
color refinement~\cite{Morris2019,Xu2019}.
We investigate the spanning structures that are nevertheless
guaranteed to be shared by graphs that color refinement cannot
distinguish. For connected graphs of fixed cyclomatic number, we
establish a sharp bound on the number of components in a common
spanning forest that depends only on the cyclomatic number, not on
the order of the graphs.

Unless explicitly stated otherwise, the graphs in our main statements
are finite, undirected, simple, and uncolored. Multigraphs arise only
as auxiliary objects in the proofs. We use uniformly initialized
color refinement, namely the one-dimensional Weisfeiler--Leman
algorithm (1-WL)~\cite{Arvind2017}. Starting with the same color at
every vertex, each round replaces a vertex's color by the pair
consisting of its previous color and the multiset of its neighbors'
colors. We write $G\equiv_{\wl}H$ if, when refinement is run on the
disjoint union $G\sqcup H$, each color occurs equally often in $G$
and $H$ at every round. This notion agrees with fractional
isomorphism~\cite{Tinhofer1986,Ramana1994,Arvind2017}.
In particular, $G\equiv_{\wl}H$ implies that $G$ and $H$ have the
same order and degree multiset, and hence the same number of edges.
For a connected graph $G$, its \emph{cyclomatic number} is
\[
r(G):=|E(G)|-|V(G)|+1.
\]
Thus connected color-refinement-equivalent graphs have the same
cyclomatic number.

We say that $G$ and $H$ have a \emph{common spanning forest}
if they contain isomorphic spanning forests.
Our main result gives a sharp upper bound on the number of
components of such forests.

\begin{theorem}\label{thm:main}
Let $G$ and $H$ be connected finite simple uncolored graphs with
$G\equiv_{\wl}H$ and common cyclomatic number $r\ge2$. Then $G$ and $H$
have a common spanning forest with at most $r-1$ components.
For every integer $r\ge2$, there is such a pair
of graphs of maximum degree at most three for which
every common spanning forest has at least $r-1$
components.
\end{theorem}

For graphs of common order $n$, Theorem~\ref{thm:main}
guarantees a vertex bijection that preserves at least
$n-r+1$ edges forming a spanning forest.
To measure edge disagreements under vertex relabeling,
we use the following distance, defined for equal-order graphs:
\begin{equation}\label{eq:metric}
\de(G,H)=\min_{\pi:V(G)\longrightarrow V(H)\ \mathrm{bijective}}
       |E(H)\mathbin{\triangle}\pi(E(G))|,
\end{equation}
where $\pi(E(G))=\{\{\pi(u),\pi(v)\}:\{u,v\}\in E(G)\}$.
The minimum is taken over all vertex bijections, with no requirement
to preserve color-refinement colors.
Let $\Dcr(r)$ denote the supremum of $\de(G,H)$ over connected
pairs $G\equiv_{\wl}H$ with $r(G)=r(H)=r$, with no restriction
on their common order.
For the connected pairs under consideration, each graph has
$n+r-1$ edges, so Theorem~\ref{thm:main} implies
$\de(G,H)\le4(r-1)$.
This upper bound is sharp for every $r\ge2$, as stated
in the following corollary.

\begin{corollary}\label{cor:diameter}
For every integer $r\ge2$,
\(
\Dcr(r)=4(r-1).
\)
\end{corollary}

For $r=0$, any two connected color-refinement-equivalent graphs
are isomorphic, since color refinement distinguishes trees up
to isomorphism~\cite[Theorem~2.5]{Ramana1994}.
The same holds for $r=1$: two connected unicyclic graphs
that are equivalent under color refinement are
isomorphic~\cite[proof of Lemma~12]{Verbitsky2025}.
Consequently, $\Dcr(0)=\Dcr(1)=0$.
Together with these results, Corollary~\ref{cor:diameter}
determines $\Dcr(r)$ for every nonnegative integer $r$.

Common-subgraph problems have been studied in several algorithmic
settings, including approximation algorithms for pairs of
forests~\cite{RautenbachWerner2026}, parameterized algorithms for
spanning forest isomorphism and containment with a tree as the
target~\cite{LiuEtAl2026}, and integer-programming formulations for
the maximum common edge subgraph problem~\cite{deGastinesKnippel2024}.
Our work instead establishes a sharp structural guarantee for common
spanning forests of connected color-refinement-equivalent graphs,
with the cyclomatic number controlling the number of components.

The associated edge symmetric-difference distance is closely related
to graph similarity and approximate isomorphism. In particular,
$\de(G,H)$ is half the minimum squared Frobenius distance between
the adjacency matrices under vertex relabelings
\cite{GroheRattanWoeginger2018}.
Dahan et al.~\cite{DahanGroheNeuenNovotny2026} study additive
$\varepsilon n^2$ approximations to graph edit distance for graphs
of bounded VC dimension, as well as the distinguishing power of
higher-dimensional Weisfeiler--Leman refinement under edit-distance
separation promises.
B\"oker's tree distance~\cite{Boker2021}, by contrast, vanishes on
fractionally isomorphic graphs, whereas $\de(G,H)$ can be positive
even when $G\equiv_{\wl}H$.
Our result determines the exact maximum of this edge distance
among connected color-refinement-equivalent graphs of fixed
cyclomatic number, without restricting their common order.

The remainder of the paper is organized as follows.
Section~\ref{sec:preliminaries} develops the kernel reduction
and a lemma for lifting spanning forests to the original graphs.
Section~\ref{sec:kernel} constructs the required kernel forests
and establishes the existence part of Theorem~\ref{thm:main}
and the upper bound in Corollary~\ref{cor:diameter}.
Section~\ref{sec:sharpness} establishes the sharpness of both
bounds. The auxiliary theorem on subcubic trees used in
Section~\ref{sec:kernel} is proved in
Appendix~\ref{app:induction}.

%% file: sections/02_preliminaries.tex
\section{The kernel reduction}
\label{sec:preliminaries}

In this section, we reduce the existence part of
Theorem~\ref{thm:main} to finding suitable spanning forests
in the kernels of two graphs.

Let $G\equiv_{\wl}H$ be two connected graphs with common
cyclomatic number $r\ge2$, and let $\chi$ be the stable
coloring obtained by running color refinement on
$G\sqcup H$. It is well known that $\chi$ is an
\emph{equitable coloring} (meaning that, for every pair
of colors $i,j$, all vertices of color $i$ have the
same number of neighbors of color $j$) and that each
color occurs equally often in $G$ and
$H$~\cite{Ramana1994,Arvind2017}.

Let $C_G:=\core(G)$ and $C_H:=\core(H)$ denote the
\emph{$2$-cores} of $G$ and $H$, respectively.
Since the $2$-core is obtained by repeatedly deleting
all vertices of degree at most one, equitability implies
that at each step all vertices of a given stable color
are deleted together. Thus equally colored vertices
either both belong to their respective $2$-cores or
both lie outside them~\cite[Claim~16]{Verbitsky2025}.
Equitability also implies that equally colored core
vertices have the same degree within their respective
cores. Let $K_G$ and $K_H$ denote the \emph{kernels}
obtained by suppressing all degree-two vertices in
$C_G$ and $C_H$~\cite{Verbitsky2025}.
Since $r\ge2$, the kernels $K_G$ and $K_H$ are connected
multigraphs, possibly with loops and parallel edges, whose
vertices are precisely the \emph{branch vertices}, namely
the vertices of degree at least three in the respective
$2$-cores. As usual, loops contribute two to vertex degrees,
and parallel edges are counted with multiplicity.
It follows that $K_G$ and $K_H$ have the
same number $b$ of vertices and the same number $q$
of edges. Removing leaves
and suppressing degree-two vertices preserve the
cyclomatic number, giving $r=q-b+1$. Since every
kernel vertex has degree at least three, we also have
$2q\ge3b$.
For each vertex $v$ in $C_G$ or $C_H$, let $T_v$ be
the component containing $v$ after deleting all edges
of the corresponding $2$-core, rooted at $v$.
Each $T_v$ is a tree containing no other core vertex.
The \emph{attached trees} $T_v$ and $T_w$ are isomorphic
as rooted trees whenever $\chi(v)=\chi(w)$~\cite[Claim~15]{Verbitsky2025}.

The kernel records the connections between branch vertices
but not the vertices suppressed along its edges.
For each incidence of a kernel edge at a branch vertex $v$,
write the corresponding path in the $2$-core as
$x_0,x_1,\ldots,x_\ell$, starting at that incidence,
so that $x_0=v$ and $x_\ell$ is also a branch vertex,
possibly with $x_0=x_\ell$.
We associate with this incidence the sequence of stable colors
\[
(\chi(x_0),\chi(x_1),\ldots,\chi(x_\ell)).
\]
The other incidence of the same kernel edge gives the
reversed sequence.
Both incidences of a loop are counted separately,
even when their sequences coincide.
For each branch vertex $v$ in either $G$ or $H$, let $M(v)$ denote the
multiset of these sequences over all incidences at $v$.
Thus $M(v)$ contains $\deg_K(v)$ entries, counted with
multiplicity, where $K$ is the kernel containing $v$.

We have the following consequence of equitability.

\begin{lemma}\label{lem:pathdata}
If $v$ and $w$ are branch vertices with
$\chi(v)=\chi(w)$, then $M(v)=M(w)$.
\end{lemma}

\begin{proof}
By equitability and the fact that core membership is
determined by stable colors, the color of a branch
vertex determines the multiset of colors of its
neighbors in the $2$-core.

Consider a core path $x_0,x_1,\ldots,x_\ell$
starting at a branch vertex $x_0$.
If $x_1$ is a branch vertex, the sequence ends there.
Otherwise, $x_1$ has degree two in the core, and its
stable color determines the multiset of colors of its
two core neighbors.
Since one of these neighbors is $x_0$, the colors
$\chi(x_0)$ and $\chi(x_1)$ uniquely determine
$\chi(x_2)$.

The same argument applies at each subsequent internal
vertex of the path.
Moreover, since core degrees are determined by stable
colors, the sequence also determines when the next
branch vertex is reached.
Thus, for a fixed color of $x_0$, the color of its
first neighbor determines the entire sequence.
As the initial neighbor colors occur with the same
multiplicities at equally colored branch vertices,
the asserted equality $M(v)=M(w)$ follows.
\end{proof}


For a spanning forest $\mathcal F$ of a kernel,
we call its edges \emph{selected} and the
remaining kernel edges \emph{unselected}.
The following lemma uses
Lemma~\ref{lem:pathdata} to lift suitably
matched spanning forests of the kernels
to the original graphs.

\begin{lemma}\label{lem:lift}
Let $\mathcal F_G$ and $\mathcal F_H$ be spanning
forests of $K_G$ and $K_H$, each with $a$ edges.
Suppose there is an isomorphism
\[
f:\mathcal F_G\longrightarrow\mathcal F_H
\]
that preserves the stable colors of the branch
vertices. Assume further that for every edge $uv$
of $\mathcal F_G$, the core paths represented by
$uv$ and $f(u)f(v)$ have the same sequence of
stable colors when read from $u$ to $v$ and from
$f(u)$ to $f(v)$, respectively.
Then $G$ and $H$ have isomorphic spanning forests
with exactly $b-a$ components, each obtained by
deleting exactly $q-a$ edges from the corresponding
graph.
\end{lemma}

\begin{proof}
For every branch vertex $v$ of $K_G$,
Lemma~\ref{lem:pathdata} gives
\(
M(v)=M(f(v)).
\)
Since $f$ matches the selected edges and their
color sequences, these edges contribute the same
sequences to $M(v)$ and $M(f(v))$.
It follows that the sequences contributed by the
unselected edges also form equal multisets.

We now construct a bipartite graph $\Gamma$ whose
vertex classes are the unselected edges of
$K_G$ and $K_H$, respectively.
Two vertices of $\Gamma$, representing edges
$e$ of $K_G$ and $e'$ of $K_H$, are adjacent
if there exists a branch vertex $v$ incident
with $e$ such that $f(v)$ is incident with
$e'$ and the corresponding core paths have
the same color sequence when read from $v$
and $f(v)$, respectively.
For a loop, both directions along its
core path are considered.

We claim that $\Gamma$ has a perfect matching.
Let $S$ be any subset of the vertex class of
$\Gamma$ corresponding to $K_G$, and let
$N_\Gamma(S)$ denote its set of neighbors.
Each vertex of $S$ represents an unselected
kernel edge of $K_G$, which has two ends.
Thus these edges have exactly $2|S|$
ends in total.
Fix a branch vertex $v$ of $K_G$ and a color
sequence. The number of ends of unselected
edges at $v$ with this sequence equals the
corresponding number at $f(v)$, as shown above.
Since $S$ contains only some of the unselected
edges, we can assign each end of an edge in
$S$ at $v$ to a distinct end of an unselected
edge at $f(v)$ having the same sequence.
Performing this assignment for every branch
vertex and every color sequence gives
$2|S|$ distinct ends of unselected edges
of $K_H$. By the definition of $\Gamma$, all these ends
belong to kernel edges represented by vertices
in $N_\Gamma(S)$.
Since these kernel edges have exactly
$2|N_\Gamma(S)|$ ends in total, we obtain
\[
2|S|\le 2|N_\Gamma(S)|.
\]
Thus $|S|\le |N_\Gamma(S)|$ for every $S$.
Since both vertex classes of $\Gamma$ have
$q-a$ vertices, Hall's theorem~\cite{Hall1935}
gives a perfect matching of $\Gamma$.

Let $\mathcal M$ be such a perfect matching.
We now construct spanning subgraphs of $C_G$ and
$C_H$ and extend $f$ to a bijection
\[
\varphi:V(C_G)\longrightarrow V(C_H).
\]
For every branch vertex $v$ of $C_G$,
set $\varphi(v):=f(v)$.
Consider an edge of $\mathcal F_G$ and
its corresponding edge of $\mathcal F_H$.
Write the core paths represented by
these selected kernel edges as
$x_0,x_1,\ldots,x_\ell$ and
$y_0,y_1,\ldots,y_\ell$, where
$y_0=f(x_0)$ and $y_\ell=f(x_\ell)$.
By the hypothesis of the lemma, these
paths have the same color sequence.
We define $\varphi(x_i):=y_i$ for
$1\le i<\ell$ and include all edges
of both paths in the respective subgraphs.

Next, consider a pair of unselected
kernel edges matched by $\mathcal M$.
By the definition of $\Gamma$, we can
write their corresponding core paths as
$x_0,x_1,\ldots,x_\ell$ and
$y_0,y_1,\ldots,y_\ell$ so that
$y_0=f(x_0)$ and
$\chi(x_i)=\chi(y_i)$ for
$0\le i\le\ell$.
For these paths, we define
$\varphi(x_i):=y_i$ for $1\le i<\ell$.
We include all edges of these paths except
the terminal edges $x_{\ell-1}x_\ell$
and $y_{\ell-1}y_\ell$ in the respective subgraphs.
This construction also applies when
$\ell=1$ or when a core path returns
to its starting branch vertex.

Since $\mathcal M$ matches every unselected
kernel edge exactly once, all core paths have
now been considered.
Every core vertex is either a branch vertex
or an internal vertex of exactly one of these
paths.
Distinct core paths have disjoint interiors,
and matched paths have the same length.
Together with the correspondence $f$ on branch
vertices, the assignments above therefore
define a bijection
\(
\varphi:V(C_G)\longrightarrow V(C_H).
\)
By construction, $\varphi$ preserves stable
colors, and the edges chosen along corresponding
core paths agree under $\varphi$.
Thus $\varphi$ is an isomorphism between the
two spanning subgraphs of the cores.

We next extend $\varphi$ to the original
graphs using the attached trees.
For every $v\in V(C_G)$, the rooted
trees $T_v$ and $T_{\varphi(v)}$
are isomorphic because
$\chi(v)=\chi(\varphi(v))$.
Choose a rooted isomorphism between
them that agrees with $\varphi$
at the root.
Since distinct attached trees meet
the cores only at their roots,
these isomorphisms extend $\varphi$
to a bijection
\[
\varphi:V(G)\longrightarrow V(H).
\]
In each original graph, we take the union of
the constructed spanning subgraph of its core
and all the attached trees.
The resulting subgraphs span $G$ and $H$,
respectively.

By construction, the resulting spanning subgraphs
of $G$ and $H$ are isomorphic under $\varphi$.
They are also forests.
Indeed, the core paths corresponding to selected
kernel edges connect the branch vertices exactly
as the edges of $\mathcal F_G$ and $\mathcal F_H$
do, and hence form forests.
For each unselected kernel edge, the retained
part of its core path meets the branch vertices
only at its starting vertex.
Thus it cannot create a cycle or join two
components of the selected forests.
The same holds for each attached tree $T_v$,
which meets the core only at its root.

Since $\mathcal F_G$ and $\mathcal F_H$ are spanning forests
on $b$ vertices with $a$ edges, each has exactly $b-a$ components.
The constructed spanning forests therefore have exactly $b-a$
components as well.

Finally, all edges of the selected core paths and the attached trees
are included in the constructed forests.
For each of the $q-a$ unselected kernel edges, exactly one edge of
its corresponding core path is omitted.
Hence each spanning forest is obtained by deleting exactly
$q-a$ edges from the original graph, as claimed.
\end{proof}

Consequently, to establish the existence part of
Theorem~\ref{thm:main}, it suffices to find spanning
forests of $K_G$ and $K_H$ satisfying the hypotheses
of Lemma~\ref{lem:lift} with $b-a\le r-1$.
Since this inequality is equivalent to
$a\ge b-r+1$, and $a\ge0$, the required bound is
\[
a\ge\max\{0,b-r+1\}=\max\{0,2b-q\},
\]
where the last equality follows from $r=q-b+1$.

%% file: sections/03_common_forests.tex
\section{Constructing forests in the kernels}
\label{sec:kernel}

In this section, we construct the spanning forests of the kernels needed
for the existence part of Theorem~\ref{thm:main}.
We first use the stable colors of branch vertices and the color sequences
of core paths to construct such forests, except in a special cubic case.
We then resolve that case using a universal forest theorem for subcubic trees.

\subsection{Color-based forest constructions}
\label{subsec:color-forests}

Index the stable colors of the branch vertices by $1,\ldots,h$.
For each $i$, let $B_i^G$ be the set of branch vertices of color $i$ in $G$,
and define $B_i^H$ similarly.
These sets have the same size because stable colors determine which vertices
are branch vertices and occur equally often in $G$ and $H$.
Let $s_i:=|B_i^G|=|B_i^H|$ denote their common size.

We first show how a bijection on one color class can be extended to another
while pairing corresponding kernel edges.

\begin{lemma}\label{lem:hall}
Let $i\ne j$, and let $\omega$ be the sequence of stable colors along a core path
represented by an edge of $K_G$ from a vertex of color $j$ to one of color $i$.
For any bijection $f_i:B_i^G\to B_i^H$, there exists a bijection
$f_j:B_j^G\to B_j^H$ with the following property.
For every $x\in B_j^G$, some $y\in B_i^G$ is joined to $x$ by an edge of $K_G$,
and $f_i(y)$ is joined to $f_j(x)$ by an edge of $K_H$, such that the corresponding
core paths both have the sequence $\omega$ when read from their color-$j$ endpoints.
\end{lemma}

\begin{proof}
If a kernel edge represents a core path with color sequence $\omega$ when
read from its endpoint of color $j$, the same path has the reversed
sequence when read from its endpoint of color $i$.
By Lemma~\ref{lem:pathdata}, the multiset $M(v)$ depends only on the stable
color of $v$, whether $v$ belongs to $G$ or $H$.
Thus every branch vertex of color $j$ is incident with exactly $k_j>0$
kernel edges whose core paths have the sequence $\omega$ when read from
that vertex.
Similarly, every branch vertex of color $i$ is incident with exactly
$k_i>0$ kernel edges whose core paths have the reversed sequence
when read from that vertex.

Construct a bipartite multigraph $R$ with vertex classes $B_j^G$ and $B_j^H$.
For $x\in B_j^G$ and $x'\in B_j^H$, add an edge between them for each triple
$(y,e,e')$ such that $y\in B_i^G$, $e$ is an edge of $K_G$ joining $x$ to $y$,
$e'$ is an edge of $K_H$ joining $x'$ to $f_i(y)$, and the core paths represented
by $e$ and $e'$ both have the sequence $\omega$ when read from $x$ and $x'$,
respectively.
For each $x\in B_j^G$, there are $k_j$ choices for $e$, each determining
$y$, followed by $k_i$ choices for $e'$ at $f_i(y)$.
Hence $\deg_R(x)=k_i k_j$.
Similarly, for each $x'\in B_j^H$, there are $k_j$ choices for $e'$,
each determining $f_i(y)$ and hence $y$, followed by $k_i$ choices
for $e$ at $y$.
Thus $\deg_R(x')=k_i k_j$, so $R$ is a regular bipartite multigraph
of positive degree.

Therefore $R$ has a perfect matching~\cite[Theorem~2.2]{AkiyamaKano2011}.
Let $f_j:B_j^G\to B_j^H$ be the bijection defined by this matching.
For each $x\in B_j^G$, the edge in the matching joining $x$ to $f_j(x)$
corresponds to a triple $(y,e,e')$ as defined above.
Thus $e$ joins $x$ to $y$, while $e'$ joins $f_j(x)$ to $f_i(y)$,
and their core paths have the sequence $\omega$ when read from
$x$ and $f_j(x)$, respectively.
The chosen kernel edges are distinct in each kernel because their
endpoints of color $j$ are distinct.
\end{proof}

We next apply Lemma~\ref{lem:hall} successively to construct spanning forests
when $h\ge2$.

\begin{proposition}\label{prop:multiple}
If $h\ge2$, there exist spanning forests of $K_G$ and $K_H$ satisfying
the hypotheses of Lemma~\ref{lem:lift}, each with exactly
$b-\min_{1\le i\le h}s_i$ edges.
\end{proposition}

\begin{proof}
Let $Q$ be the simple quotient graph obtained from $K_G$ by identifying
branch vertices of the same stable color and removing loops and parallel edges.
Thus its vertices are the colors $1,\ldots,h$, with two distinct colors adjacent
if a kernel edge joins vertices of those colors.
By Lemma~\ref{lem:pathdata}, the same quotient graph is obtained from $K_H$.
The graph $Q$ is connected because $K_G$ is connected.

Put $s_*:=\min_{1\le i\le h}s_i$. Choose a spanning tree of $Q$ and root it
at a color $i_0$ with $s_{i_0}=s_*$. Choose any bijection between
$B_{i_0}^G$ and $B_{i_0}^H$.
Starting from the root, extend the bijection along the rooted spanning tree of $Q$.
Suppose that a bijection $f_i:B_i^G\to B_i^H$ has already been defined, and
let $j$ be a child of $i$. Choose an edge of $K_G$ joining these color classes,
and let $\omega$ be the sequence of stable colors along its corresponding core
path, read from its endpoint of color $j$.
By Lemma~\ref{lem:hall}, we obtain a bijection $f_j:B_j^G\to B_j^H$ and,
for each $x\in B_j^G$, a pair of corresponding kernel edges joining $x$
and $f_j(x)$ to vertices paired by $f_i$. Their core paths have the same
color sequence when read from their endpoints of color $j$.
Repeating this step throughout the rooted tree defines a stable-color-preserving
bijection $f$ between all branch vertices.
Let $\mathcal F_G$ and $\mathcal F_H$ be the spanning subgraphs of $K_G$
and $K_H$ containing all branch vertices and the edges chosen in this construction.
Every vertex outside the root class contributes exactly one edge joining
it to its parent color class. Since each color class has a unique parent
in the rooted tree of $Q$, no edge is chosen more than once.
The root class contains $s_*$ of the $b$ branch vertices, so there are
$b-s_*$ vertices outside it. Thus each subgraph has exactly $b-s_*$ edges.

We now show that $\mathcal F_G$ and $\mathcal F_H$ are forests.
Assign each branch vertex the depth of its color class in the rooted
spanning tree of $Q$. Every chosen edge joins vertices at consecutive depths.
Suppose that either subgraph contains a cycle, and choose a vertex of maximum
depth on that cycle. Both cycle edges incident with this vertex must lead
to its parent color class, contradicting the fact that the construction
chooses exactly one such edge for each vertex outside the root class.
Hence $\mathcal F_G$ and $\mathcal F_H$ are spanning forests.
By construction, $f$ is an isomorphism between $\mathcal F_G$ and
$\mathcal F_H$ that preserves stable colors. The corresponding core
paths have the same color sequence when read from their child endpoints,
and hence also when read from their parent endpoints.
Thus the hypotheses of Lemma~\ref{lem:lift} are satisfied.
\end{proof}

Since $h\ge2$, we have $\min_{1\le i\le h}s_i\le b/2$.
Together with $2q\ge3b$, this gives
\[
b-\min_{1\le i\le h}s_i\ge 2b-q.
\]

We next consider the case $h=1$. One special case arises when both kernels
are cubic and every kernel edge represents a core path with the same sequence
of stable colors, unchanged under reversal. We call this the
\emph{uniform cubic case}. The following lemma shows that, outside this
special case, we can find spanning forests satisfying the hypotheses of
Lemma~\ref{lem:lift} with at least $\max\{0,2b-q\}$ edges each.

\begin{lemma}\label{lem:reduction}
Suppose $h=1$ and the uniform cubic case does not occur.
Then there exist spanning forests of $K_G$ and $K_H$ satisfying the
hypotheses of Lemma~\ref{lem:lift}, each with at least
$\max\{0,2b-q\}$ edges.
\end{lemma}

\begin{proof}
As shown in Section~\ref{sec:preliminaries}, equally colored branch vertices
have the same degree in their respective $2$-cores.
Since $h=1$, all branch vertices have the same stable color, so both
kernels are $d$-regular for some $d\ge3$. Hence $2q=db$ and $2b-q=(4-d)b/2$.
If $d\ge4$, then $\max\{0,2b-q\}=0$, and the edgeless spanning forests
of $K_G$ and $K_H$ trivially satisfy the hypotheses of Lemma~\ref{lem:lift}.

Suppose now $d=3$. By Lemma~\ref{lem:pathdata}, the multiset $M(v)$ is
the same at every branch vertex of both kernels. 
The two ends of each kernel edge contribute mutually reversed color
sequences. Thus any sequence and its reversal occur equally often among
all ends of the kernel edges. As $M(v)$ is the same at every branch vertex,
they also occur with equal multiplicities in $M(v)$.

Each multiset $M(v)$ contains exactly three entries, as $d=3$.
Any sequence that differs from its reversal must occur together with
that reversal, with the same multiplicity. If such a pair occurs,
the remaining entry must be a sequence unchanged under reversal,
and it occurs exactly once in $M(v)$.
Otherwise, all three entries are unchanged under reversal.
If they are not all identical, at least one sequence occurs exactly once.
If all three entries are identical, then every kernel edge in both kernels
represents a core path with the same sequence, unchanged under reversal,
because $M(v)$ is the same at every branch vertex.
This is precisely the uniform cubic case, which is excluded by hypothesis.

We may therefore choose a sequence $\omega$ that is unchanged under
reversal and occurs exactly once in $M(v)$ at every branch vertex of
both kernels. In each kernel, consider all edges whose corresponding
core paths have the sequence $\omega$ when read from one endpoint.
Since $\omega$ is unchanged under reversal, each such edge contributes
one occurrence of $\omega$ at each of its two ends.
As $\omega$ occurs exactly once in every $M(v)$, each branch vertex
has exactly one end belonging to these edges.
In particular, none of these edges is a loop, and they form a
perfect matching in each kernel.

Pair the edges of the two perfect matchings arbitrarily and match
their endpoints in either order. Since every branch vertex belongs
to exactly one matching edge, this defines an isomorphism between
the two spanning matchings that preserves stable colors.
The corresponding core paths have the same color sequence in either
direction, and each matching has $b/2=2b-q$ edges.
Thus these matchings are spanning forests satisfying the hypotheses
of Lemma~\ref{lem:lift} with at least $\max\{0,2b-q\}$ edges each,
as claimed.
\end{proof}

\subsection{The uniform cubic case}
\label{subsec:uniform-cubic}

We now consider the uniform cubic case. All branch vertices have the same
stable color, and every kernel edge represents a core path with the same
color sequence, unchanged under reversal.
Thus any isomorphism between spanning forests of $K_G$ and $K_H$
satisfies the hypotheses of Lemma~\ref{lem:lift}.

Since both kernels are cubic, we have $2q=3b$, so $b$ is even and
the lower bound on the number of edges established in
Section~\ref{sec:preliminaries} becomes
\[
\max\{0,2b-q\}=2b-\frac{3b}{2}=\frac b2.
\]
It therefore suffices to find a single forest on $b$ vertices with $b/2$ edges
that occurs as a spanning subgraph of every connected cubic multigraph
on $b$ vertices.

For $j\ge1$, let $P_j$ denote the path on $j$ vertices.
In particular, $P_1$ is an isolated vertex. 
We use sums of graphs to denote their disjoint union.
For every even integer $n\ge2$, define $k_n:=\left\lfloor (n+2)/6 \right\rfloor$ and
\begin{equation}\label{eq:Fn}
F_n:=k_nP_3+\left(\frac n2-2k_n\right)P_2+k_nP_1.
\end{equation}
All coefficients are nonnegative, and $F_n$ has exactly $n$ vertices
and $n/2$ edges. The forest $F_n$ has the following property.

\begin{theorem}\label{thm:universal}
Every tree of maximum degree at most three and even order $n\ge2$
contains $F_n$ as a spanning subgraph.
\end{theorem}

The proof is given in Appendix~\ref{app:induction}.
The theorem also gives the following result for cubic multigraphs.

\begin{corollary}\label{cor:cubic}
Every connected cubic multigraph on $n$ vertices, allowing loops
and parallel edges, contains a spanning subgraph isomorphic to $F_n$.
\end{corollary}

\begin{proof}
Let $K$ be such a multigraph. Since $3n=2|E(K)|$, the number $n$ is even
and hence at least two. Deleting loops and retaining only one edge from
each parallel class gives a connected simple spanning subgraph of $K$
with maximum degree at most three.
Choose a spanning tree $T$ of this subgraph.
By Theorem~\ref{thm:universal}, $T$ contains a spanning subgraph
isomorphic to $F_n$. Since $T$ is also a spanning subgraph of $K$,
the corollary follows.
\end{proof}

Notice that the existence of $F_n$ for simple cubic graphs, and more generally
for connected loopless cubic multigraphs, can also be deduced from
Kaneko's path-factor theorem~\cite[Theorem~1 and
Corollaries~2 and~4]{Kaneko2003} after a component-count argument.
To the best of our knowledge, the existing path-factor results do
not cover cubic multigraphs with loops, as required here.

In the uniform cubic case, Corollary~\ref{cor:cubic} gives spanning forests
of $K_G$ and $K_H$ that are both isomorphic to $F_b$.
As observed above, any isomorphism between these forests satisfies
the hypotheses of Lemma~\ref{lem:lift}, and each forest has
$b/2=2b-q$ edges.

\subsection{Proofs of Theorem~\ref{thm:main} (existence)
and Corollary~\ref{cor:diameter} (upper bound)}
\label{subsec:upper-bound-completion}


\begin{proof}[Proof of Theorem~\ref{thm:main}, existence]
For $h\ge2$, Proposition~\ref{prop:multiple} gives spanning forests satisfying
the hypotheses of Lemma~\ref{lem:lift} with at least $2b-q$ edges.
For $h=1$, Lemma~\ref{lem:reduction} covers all cases except the uniform
cubic case, which has been resolved above.

Thus, in every case, there exist spanning forests of $K_G$ and $K_H$
satisfying the hypotheses of Lemma~\ref{lem:lift}, each with
$a\ge\max\{0,2b-q\}$ edges.
Applying that lemma, we obtain isomorphic spanning forests of $G$ and $H$
with $c=b-a$ components. Since $r=q-b+1$, we have
\[
c=b-a\le q-b=r-1.
\]
This proves the existence part of Theorem~\ref{thm:main}.
\end{proof}

\begin{proof}[Proof of Corollary~\ref{cor:diameter}, upper bound]
Let $n$ and $m$ denote the common numbers of vertices and edges of
$G$ and $H$, so that $m=n+r-1$.
The existence part of Theorem~\ref{thm:main} gives isomorphic spanning
forests with $c\le r-1$ components and hence $n-c$ edges.
An isomorphism between these forests gives a bijection between the full
vertex sets that preserves at least $n-c$ edges. Therefore,
\begin{equation}\label{eq:forestmetric}
\de(G,H)\le 2\bigl(m-(n-c)\bigr)
          =2(r-1+c)\le4(r-1).
\end{equation}
This proves the upper bound in Corollary~\ref{cor:diameter}.
\end{proof}

%% file: sections/04_sharpness.tex
\section{Sharpness}\label{sec:sharpness}

In this section, we establish the matching lower bound in
Corollary~\ref{cor:diameter} by constructing, for every $r\ge2$,
connected simple graphs $G,H$ with common cyclomatic number $r$
such that $G\equiv_{\wl}H$ and $\de(G,H)=4(r-1)$.
The same examples also show that the component bound in
Theorem~\ref{thm:main} is sharp.

\subsection{The case $r\ge3$}
\label{subsec:sharpness-ge3}

Set $s:=r-1\ge2$, and take all indices modulo $s$.
Let
\[
V_s:=\{a_i,b_i,x_i,y_i,z_i,p_i,q_i,j_i,u_i,v_i:
0\le i<s\},
\]
and define
\begin{equation}\label{eq:common-edges}
E_0:=\bigcup_{i=0}^{s-1}
\{a_ix_i,\ x_iy_i,\ y_iz_i,\ b_ip_i,\ p_iq_i,
  a_ij_i,\ j_ib_i,\ p_iu_i,\ q_iv_i\}.
\end{equation}
Define graphs $G_s$ and $H_s$ on $V_s$ by
\begin{equation}\label{eq:lowerpattern}
\begin{aligned}
E(G_s)&=E_0\cup\{z_ia_i,\ q_ib_{i+1}:0\le i<s\},\\
E(H_s)&=E_0\cup\{z_ia_{i+1},\ q_ib_i:0\le i<s\}.
\end{aligned}
\end{equation}
Figure~\ref{fig:sharpness-construction} illustrates the case $s=2$.

\begin{figure}[htbp]
\centering
\includegraphics[width=0.80\textwidth]
{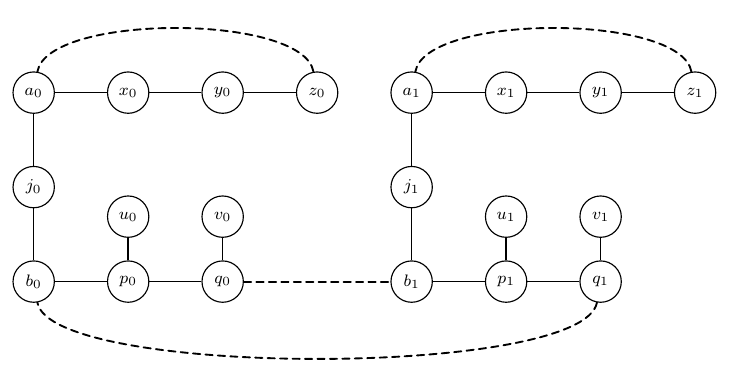}

\smallskip
$G_2$

\medskip
\includegraphics[width=0.80\textwidth]
{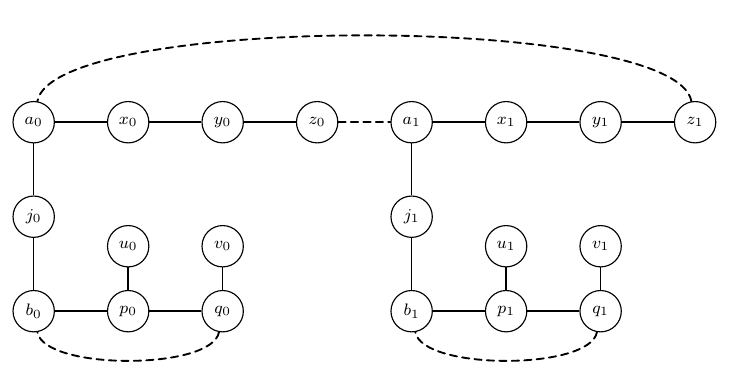}

\smallskip
$H_2$

\caption{The graphs $G_2$ and $H_2$.
Solid edges belong to $E_0$, while dashed edges
are the additional edges defining each graph.}
\label{fig:sharpness-construction}
\end{figure}

By construction, $G_s$ and $H_s$ are connected simple
cactus graphs of maximum degree three.
Both have $10s$ vertices and $11s$ edges, so their
common cyclomatic number is
\(
11s-10s+1=s+1=r.
\)

We first verify that these two graphs are indistinguishable by
uniformly initialized color refinement.

\begin{lemma}\label{lem:lowerwl}
$G_s\equiv_{\wl}H_s$.
\end{lemma}

\begin{proof}
Consider the partition of $V_s$ into the seven vertex
classes listed below, where $0\le i<s$.
The last column gives the multiset of neighboring classes
in either graph.
\begin{center}
\begin{tabular}{@{}clcc@{}}
\toprule
Class & Vertices & Size & Neighbor classes\\
\midrule
$A$ & $a_i$       & $s$  & $X,X,J$\\
$B$ & $b_i$       & $s$  & $P,P,J$\\
$X$ & $x_i,z_i$   & $2s$ & $A,Y$\\
$Y$ & $y_i$       & $s$  & $X,X$\\
$P$ & $p_i,q_i$   & $2s$ & $B,P,L$\\
$J$ & $j_i$       & $s$  & $A,B$\\
$L$ & $u_i,v_i$   & $2s$ & $P$\\
\bottomrule
\end{tabular}
\end{center}
The partition is equitable in both graphs, with identical
class sizes and neighbor counts. Hence
$G_s\equiv_{\wl}H_s$ by
\cite[Theorem~2.2]{Ramana1994}.
\end{proof}

We now determine the unlabeled edge symmetric-difference distance
between $G_s$ and $H_s$.

\begin{proposition}\label{prop:lower}
For every $s\ge2$,
\(
\de(G_s,H_s)=4s.
\)
\end{proposition}

\begin{proof}
Let $S_{2,2}$ denote the tree consisting of two adjacent
vertices of degree three, each joined to two leaves.
For each $0\le i<s$, let $Q_i$ and $T_i$ be subgraphs
of $G_s$ with edge sets
\[
\begin{aligned}
E(Q_i)&=\{a_ix_i,\ x_iy_i,\ y_iz_i,\ z_ia_i\},\\
E(T_i)&=\{p_iq_i,\ p_ib_i,\ p_iu_i,\ q_ib_{i+1},\ q_iv_i\}.
\end{aligned}
\]
Then $Q_i$ is a $4$-cycle and $T_i$ is isomorphic to
$S_{2,2}$. These $2s$ subgraphs are pairwise edge-disjoint.

We show that $H_s$ contains neither a $4$-cycle
nor a copy of $S_{2,2}$.
The only cycles in $H_s$ have lengths $3$ and $4s$.
Since $s\ge2$, $H_s$ contains no $4$-cycle.
Moreover, in a graph of maximum degree at most three,
the central edge of a copy of $S_{2,2}$ must join two
degree-three vertices with no common neighbor.
Every edge of $H_s$ joining two degree-three vertices
belongs to one of the triangles $b_i-p_i-q_i-b_i$,
whose vertices are pairwise adjacent.
Thus $H_s$ contains no copy of $S_{2,2}$.

Let $\pi:V(G_s)\to V(H_s)$ be any bijection.
At least one edge of each $Q_i$ and $T_i$ is not
preserved under $\pi$.
As these $2s$ subgraphs are pairwise edge-disjoint,
at least $2s$ edges are lost.
The two graphs have the same number of edges, so
\(
|E(H_s)\mathbin{\triangle}\pi(E(G_s))|\ge4s.
\)
Taking the minimum over $\pi$ gives
$\de(G_s,H_s)\ge4s$.

For the reverse inequality, the identity bijection
preserves exactly the edges of $E_0$.
By \eqref{eq:lowerpattern}, each graph has exactly
$2s$ edges outside $E_0$, so
\(
\de(G_s,H_s)
\le |E(G_s)\mathbin{\triangle}E(H_s)|=4s.
\)
Hence $\de(G_s,H_s)=4s$.
\end{proof}

\subsection{The case $r=2$}
\label{subsec:sharpness-r2}

For $r=2$, we use the six-vertex pair shown in
Figure~\ref{fig:sharpness-r2}, which also appears, up to vertex
labels, in Bodnar et al.~\cite[supplementary material,
Figure~8 (left)]{Bodnar2021}.
The graphs $G$ and $H$ have the common vertex set
$\{0,1,2,3,4,5\}$.

\begin{figure}[h]
\centering
\includegraphics[width=0.80\textwidth]{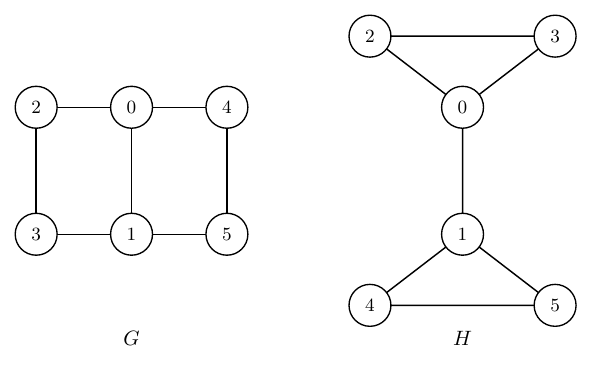}
\caption{The graphs $G$ and $H$ used for $r=2$.
The graph $G$ consists of two $4$-cycles sharing the edge $01$,
whereas $H$ consists of two triangles joined by the edge $01$.}
\label{fig:sharpness-r2}
\end{figure}

\begin{proposition}\label{prop:lower-r2}
The graphs $G$ and $H$ in Figure~\ref{fig:sharpness-r2}
are connected, have maximum degree three and cyclomatic
number $2$, satisfy $G\equiv_{\wl}H$, and have
$\de(G,H)=4$.
\end{proposition}

\begin{proof}
Both graphs are connected, have maximum degree three 
and have six vertices and seven edges,
so their cyclomatic number is $7-6+1=2$.
In each graph, the partition
$\{\{0,1\},\{2,3,4,5\}\}$ is equitable: every vertex in the
first class has one neighbor in that class and two in the second,
while every vertex in the second class has one neighbor in each
class. Since the class sizes and neighbor counts agree in both
graphs, $G\equiv_{\wl}H$ under uniform initialization by
\cite[Theorem~2.2]{Ramana1994}.

The graph $G$ is bipartite, with parts $\{0,3,5\}$ and
$\{1,2,4\}$, whereas $H$ contains the edge-disjoint triangles
$023$ and $145$.
Under any vertex bijection $\pi:V(G)\to V(H)$, the graph
$\pi(G)$ is bipartite, so it omits at least one edge of each
triangle in $H$. Since the two triangles are edge-disjoint and
$G,H$ have the same number of edges, this gives
$\de(G,H)\ge4$.
On the other hand, the identity bijection preserves exactly
the five edges of the spanning path $3-2-0-1-5-4$,
and therefore has symmetric-difference cost four.
Hence $\de(G,H)=4$.
\end{proof}

\subsection{Proofs of Corollary~\ref{cor:diameter} (lower bound)
and Theorem~\ref{thm:main} (sharpness)}
\label{subsec:sharpness-completion}

\begin{proof}[Proof of Corollary~\ref{cor:diameter}, lower bound]
For every $r\ge3$, the construction in
Subsection~\ref{subsec:sharpness-ge3}, with $s=r-1$,
gives two connected simple graphs with common cyclomatic number $r$
that are equivalent under color refinement by
Lemma~\ref{lem:lowerwl} and have distance $4(r-1)$ by
Proposition~\ref{prop:lower}.
For $r=2$, Proposition~\ref{prop:lower-r2} gives the same
distance equality. Together with the upper bound
\eqref{eq:forestmetric}, these examples show that
$\Dcr(r)=4(r-1)$ for every $r\ge2$.
\end{proof}

\begin{proof}[Proof of Theorem~\ref{thm:main}, sharpness]
Fix $r\ge2$ and choose one of the equality pairs $G,H$
constructed above. Both graphs have maximum degree three.
Let $n$ and $m$ denote their common order and size, so that
$m=n+r-1$.
Suppose they have a common spanning forest with $c$ components.
Its isomorphism preserves at least $n-c$ edges, and hence
\[
4(r-1)=\de(G,H)\le 2\bigl(m-(n-c)\bigr)
=2(r-1+c).
\]
Therefore $c\ge r-1$. The existence part of
Theorem~\ref{thm:main} gives a common spanning forest
with at most $r-1$ components, so the component bound is attained.
\end{proof}

Together with the existence and upper-bound proofs in
Subsection~\ref{subsec:upper-bound-completion}, this completes
the proofs of Theorem~\ref{thm:main} and
Corollary~\ref{cor:diameter}.

%% file: appendices/A_subcubic_tree_induction.tex
\section{Proof of Theorem~\ref{thm:universal}}
\label{app:induction}

For every even integer $n\ge8$, the definition of $F_n$ in \eqref{eq:Fn} gives
\begin{equation}\label{eq:Fn-recursion}
F_n=F_{n-6}+P_3+P_2+P_1.
\end{equation}
We prove Theorem~\ref{thm:universal} by induction on the order
of the tree, using the following six-vertex pruning lemma.

\begin{lemma}\label{lem:six-vertex-pruning}
Let $T$ be a subcubic tree on at least six vertices.
There is a set $S\subseteq V(T)$ of six vertices such that the induced subgraph
$T[S]$ contains vertex-disjoint copies of $P_3$ and $P_2$,
and $T-S$ is either empty or connected.
\end{lemma}

\begin{proof}
We first observe that every subcubic tree on at least five vertices
contains vertex-disjoint copies of $P_3$ and $P_2$.
Indeed, delete leaves successively until five vertices remain.
The resulting tree is either $P_5$ or is obtained by subdividing
one edge of $K_{1,3}$ once. Both contain a spanning $P_3+P_2$.

Now root $T$ at a leaf. For each vertex $v$, let $T_v$ denote
the subtree consisting of $v$ and its descendants.
Choose a vertex $v$ of maximum depth subject to
$|V(T_v)|\ge5$.
Since each vertex has at most two children and each child
subtree of $v$ has at most four vertices,
\(
s:=|V(T_v)|\in\{5,6,7,8,9\}.
\)
Since the root of $T$ is a leaf, it has a unique child whose
descendant subtree contains $|V(T)|-1\ge5$ vertices.
Hence $v$ is not the root.

If $s=5$, let $R:=T-V(T_v)$.
The tree $R$ is nonempty. Choose its unique vertex if
$|V(R)|=1$, and a leaf otherwise.
Write $w$ for the chosen vertex, and take
$S:=V(T_v)\cup\{w\}$.
Then $T-S=R-w$ is either empty or connected, and the
five-vertex observation supplies the required two paths.

If $s=6$, take $S:=V(T_v)$. Its complement is nonempty and
connected, and the same observation supplies the two paths.

If $s=7$, the two child subtrees of $v$ have orders $(2,4)$ or
$(3,3)$. Let their roots be $x$ and $y$, and take
\(
S:=V(T_x)\cup V(T_y).
\)
For orders $(2,4)$, the order-four subtree contains a $P_3$,
and the order-two subtree is a $P_2$.
For orders $(3,3)$, choose a $P_3$ from one subtree and an
edge from the other.

If $s=8$, the child subtrees have orders three and four.
Let $x$ and $y$ be their respective roots, and take
\(
S:=V(T_x)\cup\bigl(V(T_y)\setminus\{y\}\bigr).
\)
The subtree $T_x$ is a $P_3$, while the three vertices of
$T_y-y$ contain an edge, since $T_y$ has three edges and
$y$ has at most two children.

Finally, suppose $s=9$. Both child subtrees have order four;
let their roots be $x$ and $y$.
If either root has only one child, assume without loss of
generality that it is $x$, and take
\(
S:=\bigl(V(T_x)\setminus\{x\}\bigr)
   \cup\bigl(V(T_y)\setminus\{y\}\bigr).
\)
The three vertices of $T_x-x$ form a $P_3$, while those of
$T_y-y$ contain an edge by the preceding argument.
Otherwise, both $x$ and $y$ have two children, with
child-subtree orders one and two.
Let $z$ be the child of $y$ whose subtree has order two,
and take
\(
S:=V(T_x)\cup V(T_z).
\)
The subtree $T_x$ contains a $P_3$, and $T_z$ is a $P_2$.

For $s=5$, we have already shown that $T-S$ is either empty
or connected. In every case with $s\ge6$, the set $S$ is a
union of complete descendant subtrees and does not contain
the original root. Thus the path from any vertex of $T-S$
to the root avoids $S$, so $T-S$ is nonempty and connected.
In all cases, the chosen set $S$ has six vertices, and $T[S]$
contains vertex-disjoint copies of $P_3$ and $P_2$.
Hence $S$ has all the required properties.
\end{proof}

\begin{proof}[Proof of Theorem~\ref{thm:universal}]
The case $n=2$ is immediate, since $F_2=P_2$.
Every tree on four vertices contains a $P_3$, and hence
contains $F_4=P_3+P_1$.
For $n=6$, Lemma~\ref{lem:six-vertex-pruning} gives
a spanning copy of $F_6=P_3+P_2+P_1$.

Let $n\ge8$ be even, and apply
Lemma~\ref{lem:six-vertex-pruning} to obtain $S$. 
The graph $T-S$ is a subcubic tree of even order
$n-6$, so by induction it contains a spanning copy of $F_{n-6}$.
On $S$, retain a $P_3$ and a disjoint $P_2$, leaving the sixth
vertex isolated. Discarding all other edges gives a spanning
forest isomorphic to $F_n$ by \eqref{eq:Fn-recursion}.
\end{proof}

%% file: CSF-1-WL.bib
@article{Hall1935,
  author = {Hall, Philip},
  title = {On Representatives of Subsets},
  journal = {Journal of the London Mathematical Society},
  volume = {s1-10},
  number = {1},
  pages = {26--30},
  year = {1935},
  doi = {10.1112/jlms/s1-10.37.26}
}

@article{Tinhofer1986,
  author = {Tinhofer, Gottfried},
  title = {Graph isomorphism and theorems of {Birkhoff} type},
  journal = {Computing},
  volume = {36},
  number = {4},
  pages = {285--300},
  year = {1986},
  doi = {10.1007/BF02240204}
}

@article{Ramana1994,
  author = {Ramana, Motakuri V. and Scheinerman, Edward R. and Ullman, Daniel},
  title = {Fractional isomorphism of graphs},
  journal = {Discrete Mathematics},
  volume = {132},
  number = {1--3},
  pages = {247--265},
  year = {1994},
  doi = {10.1016/0012-365X(94)90241-0}
}

@article{Arvind2017,
  author = {Arvind, Vikraman and K{\"o}bler, Johannes and Rattan, Gaurav and Verbitsky, Oleg},
  title = {Graph Isomorphism, Color Refinement, and Compactness},
  journal = {Computational Complexity},
  volume = {26},
  number = {3},
  pages = {627--685},
  year = {2017},
  doi = {10.1007/s00037-016-0147-6}
}

@inproceedings{Verbitsky2025,
  author = {Verbitsky, Oleg and Zhukovskii, Maksim},
  title = {Canonical Labeling of Sparse Random Graphs},
  booktitle = {42nd International Symposium on Theoretical Aspects of Computer Science (STACS 2025)},
  series = {Leibniz International Proceedings in Informatics (LIPIcs)},
  volume = {327},
  pages = {75:1--75:20},
  year = {2025},
  doi = {10.4230/LIPIcs.STACS.2025.75},
  publisher = {Schloss Dagstuhl -- Leibniz-Zentrum f{\"u}r Informatik},
  note = {Article 75}
}

@article{Kaneko2003,
  author = {Kaneko, Atsushi},
  title = {A necessary and sufficient condition for the existence of a path factor every component of which is a path of length at least two},
  journal = {Journal of Combinatorial Theory, Series B},
  volume = {88},
  number = {2},
  pages = {195--218},
  year = {2003},
  doi = {10.1016/S0095-8956(03)00027-3}
}

@inproceedings{GroheRattanWoeginger2018,
  author = {Grohe, Martin and Rattan, Gaurav and Woeginger, Gerhard J.},
  title = {Graph Similarity and Approximate Isomorphism},
  booktitle = {43rd International Symposium on Mathematical Foundations of Computer Science (MFCS 2018)},
  series = {Leibniz International Proceedings in Informatics (LIPIcs)},
  volume = {117},
  pages = {20:1--20:16},
  year = {2018},
  doi = {10.4230/LIPIcs.MFCS.2018.20},
  publisher = {Schloss Dagstuhl -- Leibniz-Zentrum f{\"u}r Informatik},
  note = {Article 20}
}

@inproceedings{Bodnar2021,
  author = {Bodnar, Cristian and Frasca, Fabrizio and Otter, Nina and Wang, Yu Guang and Li{\`o}, Pietro and Mont{\'u}far, Guido and Bronstein, Michael},
  title = {{Weisfeiler} and {Lehman} Go Cellular: {CW} Networks},
  booktitle = {Advances in Neural Information Processing Systems},
  volume = {34},
  pages = {2625--2640},
  year = {2021},
  url = {https://proceedings.neurips.cc/paper_files/paper/2021/hash/157792e4abb490f99dbd738483e0d2d4-Abstract.html},
  eprint = {2106.12575v3},
  archivePrefix = {arXiv},
  primaryClass = {cs.LG}
}

@article{RautenbachWerner2026,
  author = {Rautenbach, Dieter and Werner, Florian},
  title = {Largest Common Subgraph of Two Forests},
  journal = {Journal of Graph Algorithms and Applications},
  volume = {30},
  number = {1},
  pages = {47--63},
  year = {2026},
  doi = {10.7155/jgaa.v30i1.2967}
}

@article{LiuEtAl2026,
  author = {Liu, Jingyi and Chen, Xian and Zheng, Yicheng and Wang, Jianxin and Shi, Feng},
  title = {Parameterized algorithms for the spanning forest isomorphism and containment on tree},
  journal = {Theoretical Computer Science},
  volume = {1061},
  year = {2026},
  eid = {115652},
  note = {Article 115652},
  doi = {10.1016/j.tcs.2025.115652}
}

@article{deGastinesKnippel2024,
  author = {de Gastines, Etienne and Knippel, Arnaud},
  title = {Formulations for the maximum common edge subgraph problem},
  journal = {Discrete Applied Mathematics},
  volume = {346},
  pages = {115--130},
  year = {2024},
  doi = {10.1016/j.dam.2023.11.044}
}

@inproceedings{Boker2021,
  author = {B{\"o}ker, Jan},
  title = {Graph Similarity and Homomorphism Densities},
  booktitle = {48th International Colloquium on Automata, Languages, and Programming (ICALP 2021)},
  series = {Leibniz International Proceedings in Informatics (LIPIcs)},
  volume = {198},
  pages = {32:1--32:17},
  year = {2021},
  note = {Article 32},
  doi = {10.4230/LIPIcs.ICALP.2021.32},
  publisher = {Schloss Dagstuhl -- Leibniz-Zentrum f{\"u}r Informatik}
}

@book{AkiyamaKano2011,
  author    = {Jin Akiyama and Mikio Kano},
  title     = {Factors and Factorizations of Graphs:
               Proof Techniques in Factor Theory},
  series    = {Lecture Notes in Mathematics},
  volume    = {2031},
  publisher = {Springer},
  year      = {2011},
  doi       = {10.1007/978-3-642-21919-1}
}

@misc{DahanGroheNeuenNovotny2026,
  author = {Dahan, Anatole and Grohe, Martin and Neuen, Daniel and Novotn{\'y}, Tom{\'a}{\v{s}}},
  title = {Robust Graph Isomorphism, Quadratic Assignment and {VC} Dimension},
  year = {2026},
  eprint = {2604.12584},
  note = {arXiv:2604.12584},
  archivePrefix = {arXiv},
  url = {https://arxiv.org/abs/2604.12584}
}

@inproceedings{Morris2019,
  author = {Christopher Morris and Martin Ritzert and Matthias Fey
            and William L. Hamilton and Jan Eric Lenssen
            and Gaurav Rattan and Martin Grohe},
  title = {Weisfeiler and {L}eman Go Neural: Higher-Order Graph Neural Networks},
  booktitle = {Proceedings of the AAAI Conference on Artificial Intelligence},
  volume = {33},
  pages = {4602--4609},
  year = {2019},
  doi = {10.1609/aaai.v33i01.33014602}
}

@inproceedings{Xu2019,
  author = {Keyulu Xu and Weihua Hu and Jure Leskovec and Stefanie Jegelka},
  title = {How Powerful are Graph Neural Networks?},
  booktitle = {International Conference on Learning Representations},
  year = {2019},
  url = {https://openreview.net/forum?id=ryGs6iA5Km}
}
